\let\olddiv\div

\documentclass[a4paper, 11pt, final,underruledhead]{paper}
\usepackage[mathscr]{euscript}
\usepackage{enumerate, amssymb, math}

\begin{document}

%%%%%%%%%%%%%%%%%%%%%%%%%%%%%%%%%%%%%%%%%%%%%%%%%%%%%%%%%%%%%%%%%%%%%%%%%%%%%%%%%%%%
%%%%%%%%%%%%%%%%%%%%%%%%%%%%%%%%%%%%%%%%%%%%%%%%%%%%%%%%%%%%%%%%%%%%%%%%%%%%%%%%%%%%
%%%%%%%%%%%%%%%%%%%%%%%%%%%%%%%%%%%%%%%%%%%%%%%%%%%%%%%%%%%%%%%%%%%%%% PRIVATE TEX MACROS
%%%%%%%%%%%%%%%%%%%%%%%%%%%%%%%%%%%%%%%%%%%%%%%%%%%%%%%%%%%%%%%%%%%%%%%%%%%%%%%%%%%%
%%%%%%%%%%%%%%%%%%%%%%%%%%%%%%%%%%%%%%%%%%%%%%%%%%%%%%%%%%%%%%%%%%%%%%%%%%%%%%%%%%%%

%%%%%%%%% `FONTS'

\let\goth\mathfrak

%%%%%%%%% ENVIRONMENTS
%%%%%%%%%%%%%%%%% `theorems'

\def\myend{{}\hfill{\small$\bigcirc$}}

\newenvironment{ctext}{%
  \par
  \smallskip
  \centering
}{%
 \par
 \smallskip
 \csname @endpetrue\endcsname
}

\def\id{\mathrm{id}}

\def\konftyp(#1,#2,#3,#4){\left( {#1}_{#2}\, {#3}_{#4} \right)}
\def\binokonf(#1){\konftyp({\binom{{#1}}{2}},{{#1}-2},\binom{{#1}}{3},3)}
\let\binoconf\binokonf

\newcount\liczbaa
\newcount\liczbab

\def\binkonfo(#1,#2){\liczbaa=#2 \liczbab=#2 \advance\liczbab by -2
\def\doa{\ifnum\liczbaa = 0\relax \else
\ifnum\liczbaa < 0 \the\liczbaa \else +\the\liczbaa\fi\fi}
\def\dob{\ifnum\liczbab = 0\relax \else
\ifnum\liczbab < 0 \the\liczbab \else +\the\liczbab\fi\fi}
\konftyp(\binom{#1\doa}{2},#1\dob,\binom{#1\doa}{3},3) }

\newcount\liczbac
\def\binkonf(#1,#2){\liczbac=#2 
\def\docc{\ifnum\liczbac = 0 \relax \else\ifnum\liczbac < 0\the\liczbac 
\else+\the\liczbac\fi\fi}
B_{#1{\docc}}}

%%%%%%%%%%%%%%%%%%%%%%%%%%%%%%%%%% SYMBOLS, NOTATIONS

\def\mutyp{{\mbox{\boldmath$\mu$}}}

%%%%%%%%%%%%%%%%%%%%%%%%%%%%%%%%%% PAPER'S DATA

\title[A $\mu$-classification of simple graphs]{%
Relative complements and a `switch'-classification of simple graphs}

\author{El{\.z}bieta B{\l}aszko, %
Ma{\l}gorzata Pra{\.z}mowska, %
%% \footnote{{\em Corresponding Author}, e-mail: {\ttfamily malgpraz@math.uwb.edu.pl},
%% telephone: +48 606 74 64 56}, 
Krzysztof Pra{\.z}mowski}

%%%%%%%%%%%%%%%%%%%%%%%%%%%%%%%%%%%%%%%%%%%%%%%%%%%%%%%%%%%%%%%%%%%%%%%%%%%%%%%%%%%%
%%%%%%%%%%%%%%%%%%%%%%%%%%%%%%%%%%%%%%%%%%%%%%%%%%%%%%%%%%%%%%%%%%%%%%%%%%%%%%%%%%%%
%%%%%%%%%%%%%%%%%%%%%%%%%%%%%%%%%%%%%%%%%%%%%%%%%%%%%%%%%%%%%%% BEGINNING OF THE TEXT PROPER
%%%%%%%%%%%%%%%%%%%%%%%%%%%%%%%%%%%%%%%%%%%%%%%%%%%%%%%%%%%%%%%%%%%%%%%%%%%%%%%%%%%%
%%%%%%%%%%%%%%%%%%%%%%%%%%%%%%%%%%%%%%%%%%%%%%%%%%%%%%%%%%%%%%%%%%%%%%%%%%%%%%%%%%%%

\maketitle

\begin{abstract}
  In the paper we introduce and study a classification of finite
  (simple, undirected, loopless) graphs with respect to a switch-equivalence
  (`local-complement' equivalence of \cite{pascvebl}, an analogue of the complement-equivalence of \cite{conell}).
  In the paper we propose a simple inductive method to compute the number of switch-types of graphs on $n$ vertices
  and we show that there are exactly 16 such types of graphs on 6 vertices.
  \par\noindent
  {\it keywords}: complete graph, bipartite graph, local complement (in a point).
\end{abstract}

%%%%%%%%%%%%%%%%%%%%%%%%%%%%%%%%%%%%%%%%%%%%%%%%%%%%%%%%%%
%%%%%%%%%%%%%%%%%%%%%%%%%%%%%%%%%%%%%%%%%%%%%%%%%%%%%%%%%% sec:intro
%%%%%%%%%%%%%%%%%%%%%%%%%%%%%%%%%%%%%%%%%%%%%%%%%%%%%%%%%%
\section*{Introduction}\label{sec:intro}

In the paper we introduce and study a classification of 
(simple, undirected, loopless) graphs which was defined, in fact, 
many years ago as a convenient tool to classify
configurations of some sort (see \cite{pascvebl}, \cite{mveb2proj}).
Namely such a graph is used as a parameter of construction of a so called
multiveblen configuration. And the resulting configurations are isomorphic
when corresponding graphs are equivalent exactly in the sense considered
in this paper. 
However, at its origins, the equivalence in question was just an auxiliary
notion. Its definition was rather (formally) complicated, though it was  easy
to use when one wanted to check `by hand' whether two given graphs 
are equivalent. This original definition, now considered rather as a 
criterion,  turned out to be equivalent to a very elegant one, which can be briefly
presented as follows.
It is a folklore that the binary symmetric difference operation
determines the structure of an abelian group on the family 
$\sub({\cal X})$ of all subsets of a fixed set $\cal X$, for every $\cal X$.
In particular, we can take $\cal X$ to be the set of edges of a complete
graph $K_X$. The set of all complete bipartite subgraphs of $K_X$
yields a subgroup, isomorphic to $2^{|X|-1}$. 
And, finally, two graphs defined on $X$ are equivalent when they are congruent modulo
the group of complete bipartite subgraphs defined on $X$.
\par\medskip
In our opinion, there are two arguments which prove that this equivalence is
worth studying.
The first argument follows from the way in which it is defined: it places our notion
within investigations on natural elementary algebra of sets.
The second argument consists in interesting interpretation in terms of `switches'.
Imagine that a graph $\cal G$ characterizes possible connections between $n$ places.
Another graph ${\cal G}'$ is equivalent to $\cal G$ if it arises from $\cal G$ when in some
places all existing connections are blocked, and all the other, previously blocked,
become unblocked.
\par
This idea has appeared in the mathematical literature (especially with connections to computer
sciences) many years ago, and there is a great amount of papers where complement-equivalence,
switching-equivalence and related notions were introduced and studied in the class of {\em directed} 
graphs (digraphs). Just to quote some of them: \cite{conell}, \cite{dalhaus}, \cite{lindsey}. 
\par\medskip
Our equivalence has $2^{\binom{|X|}{2} - |X| +1}$ equivalence classes and this number 
rapidly increases when $|X|$ increases.

Quite often we are interested not in a concrete graph (a realization) but more in its isomorphism
type. The number $\mutyp_n$ of isomorphism types of respective congruence classes
grows up not so rapidly, at least for small values of 
$n = |X|$. But the exact formula for $\mutyp_n = \mutyp(n)$
is hard to find. The reasoning used to compute the size of a congruence class
cannot be applied now, as in a congruence class various numbers of pairwise nonisomorphic
graphs may appear. For example, for $n=4$, there are three iso-types of graphs equivalent
to $K_4$ and five iso-types of graphs equivalent to $L_4$.
\par
In the paper we propose an inductive method to compute $\mutyp_n$.
It is evident that $\mutyp_3 = 2$ and it was proved in \cite{pascvebl} that
$\mutyp_4 = 3$ and $\mutyp_5 = 7$.
Here we show how our machinery gives $\mutyp_6 = 16$.
This is one, particular result of the paper.
At the same time we determine fundamental general properties of the equivalence introduced,
and show several general invariants of this equivalence.

Finally, it is worth to note that the family of all complete bipartite graphs defined on $X$ together
with all disjoint unions of pairs of complete subgraphs of $K_X$ is also a subgroup of 
all subgraphs of $K_X$. Consequently, congruence modulo this subgroup also defines an
equivalence of graphs. One can note that two graphs are equivalent in this sense if either they are 
equivalent in the sense introduced in the paper or one is equivalent to the boolean completion
of the second. An interpretation of this equivalence in terms of switches is
also possible, but now one should pay attention more to a binary labeling connected/unconnected,
in fact: a labelling of the edges of $K_X$ by two distinct symbols.

%%%%%%%%%%%%%%%%%%%%%%%%%%%%%%%%%%%%%%%%%%%%%%%%%%%%%%%%%%
%%%%%%%%%%%%%%%%%%%%%%%%%%%%%%%%%%%%%%%%%%%%%%%%%%%%%%%%%% sec:defy
%%%%%%%%%%%%%%%%%%%%%%%%%%%%%%%%%%%%%%%%%%%%%%%%%%%%%%%%%%
\section{Basic definitions and facts}\label{sec:defy}

\def\syminus{\olddiv}
\newcommand{\dcup}{\mathbin{\mathaccent\cdot\cup}}
\def\grafy{{\mathcal G}}
\def\bipar{{\mathcal D}}

%%%%%%%%%%%%%%%%%%%%%%%%%%%%%%%%%%%%%%%%%%%%%%%%%%%%%%%%%%
%%%%%%%%%%%%%%%%%%%%%%%%%%%%%%%%%%%%%%%%%%%%%%%%%%%%%%%%%% ssec:graph  nortation
%%%%%%%%%%%%%%%%%%%%%%%%%%%%%%%%%%%%%%%%%%%%%%%%%%%%%%%%%%
\subsection{Graph-theoretical notations}

The equivalence of graphs investigated in the paper is closely related to a 
classification of partial Steiner triple systems of some sort.
However, the resulting classification of graphs has its own interest; 
it has quite natural intuitions and motivations concerning flows
on (undirected) graphs.

Let $X$ be a nonempty set; then $\sub(X)$  is the set of all subsets of $X$.
For an integer $k$ we write $\sub_k(X)$ for the set
of $k$-subsets of $X$.
A {\em graph} (an undirected graph without multiplied edges and loops  %%, cf. \cite{wilson})
defined on a set $X$ is an arbitrary subset $\cal E$ of $\sub_2(X)$;
if $\{x,y\}=e\in{\cal E}$ we say that $x,y$ are the {\em vertices} of the {\em edge} $e$.
More precisely, we sometimes say that a graph is the structure
$\struct{X,{\cal E}}$: if ${\cal E}\subset \sub_2(Y)$ for $Y\subsetneq X$ this 
caution is necessary.
Clearly, if $G = \struct{X,{\cal E}}$ is a graph then
  $\varkappa(G) = \struct{X,\varkappa({\cal E})}$ with $\varkappa({\cal E}) = \sub_2(X)\setminus {\cal E}$
is also a  graph defined on $X$.
Most of the notions concerning graphs used in the paper are standard and can be found in any 
standard textbook, like e.g \cite{wilson}

Recall the definition of the symmetric difference operation $\syminus$ defined on the family of sets:
$A \syminus B = (A\setminus B)\cup(B\setminus A) = (A\cup B)\setminus(A\cap B)$.
Recall also that $\syminus$ defines on each set $\sub(W)$, with $W$ arbitrary, the structure of an abelian group with $\emptyset$
as the unit and each element of order $2$.

Two  operations on graphs will be frequently used in the paper:
Let 
  $G_i = \struct{X_i,{\cal E}_i}$ for $i=1,2$.
\begin{eqnarray}
  G_1 \syminus G_2 & := & \struct{X,{\cal E}_1\syminus{\cal E}_2}
  \\ \nonumber && \strut\quad\quad\quad\text{when } X_1 = X_2 = X,
  \\
  G_1 \dcup G_2 & := & \struct{X_1\cup X_2,{\cal E}_1\cup{\cal E}_2}
  \\ \nonumber && \strut\quad\quad\quad\text{when } X_1\cap X_2 = \emptyset.
\end{eqnarray}

Several types of graphs are frequently considered in the literature;
these will also play a crucial role in forthcoming classifications.
\begin{description}\itemsep-2pt
 \item[$K_X$] $=\struct{X,\sub_2(X)}$: the {\em complete} graph on $X$;
 \item[$N_X$] $=\struct{X,\emptyset} = \varkappa(K_X)$: the {\em empty} graph
   (note that in this case `a graph' must be considered as `a structure');
 \item[$K_n$, $N_n$]: an arbitrary graph isomorphic to $K_X$, $N_X$ with $|X| = n$;
   clearly, the isomorphism type of a complete graph and of an empty graph depends only
   on the cardinality of its set of vertices.
 \item[$C_n$]: a {\em cyclic} graph on a $n$-set;
 \item[$L_n$]: a {\em linear} graph on $n$ vertices;
 \item[$K_{A,B}$] = $\struct{A\cup B,\{ \{ a,b\}\colon a\in A,b\in B \}}$ defined when $A\cap B=\emptyset$,
   $A\neq\emptyset$ or $B\neq\emptyset$, a {\em complete bipartite} graph on $A\cup B$.
 \item[$K_{n_1,n_2}$]:  a graph $K_{A,B}$ with $|A|=n_1,\,|B|=n_2$.
\end{description}
Finally, we set %% denote
\begin{eqnarray}
  \grafy(X) & := &  \left\{ \struct{X,{\cal E}}\colon {\cal E}\subset\sub_2(X) \right\},
  \\
  \bipar(X) & := &  \left\{ K_{A,X\setminus A}\colon A\subset X \right\}.
\end{eqnarray}

In most parts of the paper we shall omit proofs of elementary set theoretic
formulas, like e.g. \eqref{plus:bipar}.

%%%%%%%%%%%%%%%%%%%%%%%%%%%%%%%%%%%%%%%%%%%%%%%%%%%%%%%%%%
%%%%%%%%%%%%%%%%%%%%%%%%%%%%%%%%%%%%%%%%%%%%%%%%%%%%%%%%%% sec:equivy
%%%%%%%%%%%%%%%%%%%%%%%%%%%%%%%%%%%%%%%%%%%%%%%%%%%%%%%%%%
\subsection{Equivalences of graphs}

\def\podo{\approx}
\def\rowne{\sim}

Let us introduce the following two relations $\cong$ and $\podo$
defined on the family $\grafy(X)$.
Let $G_1,G_2\in\grafy(X)$.
\begin{eqnarray}
  G_1 \cong G_2 & \iff & G_1, G_2 \; \text{ are isomorphic},
  \\
  G_1 \podo G_2 & \iff & G_1 \syminus G_2 \in \bipar(X).
\end{eqnarray}
The following formula is valid for any $A,B\subset X$:
\begin{equation}\label{plus:bipar}
  K_{A,X\setminus A}\syminus K_{B,X\setminus B} = 
  K_{A\syminus B, X\setminus(A\syminus B)}.
\end{equation}
As an immediate consequence of \eqref{plus:bipar} and properties of the operation $\syminus$ we infer
\begin{prop}\label{prop:podo}
  The class $\bipar(X)$ is\ closed under $\syminus$. Consequently,
  the relation $\podo$ is an equivalence in $\grafy(X)$.
\end{prop}
Next, we note another set theoretical relation:
\begin{equation}\label{plus:pelne}
  K_X \syminus K_{A,X\setminus A} = K_A \dcup K_B.
\end{equation}
These two: \eqref{plus:pelne} and \eqref{plus:bipar} 
give us immediately two equivalence classes of $\podo$:
\begin{equation}\label{klasy:1}
  [K_X]_{\podo} = \left\{ K_A\dcup K_B\colon A\cup B = X, A\cap B=\emptyset \right\},\;
  [N_X]_{\podo}	= \bipar(X).
\end{equation}
One more formula of this type can be also worth to note:
\begin{equation}
 (K_A\dcup N_{X\setminus A}) \syminus K_{A,X\setminus A} = K_X\setminus K_{X\setminus A}
 = \varkappa(K_{X\setminus A}\dcup N_A),
\end{equation}
so in one $\podo$-class are: a complete subgraph and the boolean complement 
of suitable another complete subgraph.

In the original paper \cite{pascvebl} the definition of $\podo$ was introduced with the 
help of the operation of {\em local complementation}:
for $a\in X$ and ${\cal E}\in\grafy(X)$ we proceed as follows. %% define:
\begin{ctext}
   Let $e\in\sub_2(X)$. If $a\notin e$ then 
   $e\in\mu_a({\cal E})$ iff $e\in{\cal E}$.
   If $a\in e$ then
   $e\in\mu_a({\cal E})$ iff $e\notin{\cal E}$.
\end{ctext}
So, $\mu_a$ is an operation $\grafy(X)\longrightarrow\grafy(X)$;
it is called the local complementation in the point $a$.
It is seen that
\begin{equation}
  \mu_a(G) = G\syminus K_{\{a\},X\setminus\{a\}}.
\end{equation}
Consequently, with the help of \eqref{plus:bipar}, 
we arrive to the definition introduced in \cite{pascvebl}:
\begin{prop}\label{prop:locale}
  Let $G_1,G_2\in\grafy(X)$.
  $G_1\podo G_2$ iff there is a sequence $a_1,\ldots,a_k$ of elements of $X$ such that
  $G_2 = \mu_{a_k}(\ldots \mu_{a_1}(G_1)\ldots )$.
\end{prop}
Clearly, $\cong$ is also an equivalence relation on $\grafy(X)$.
But $\podo$ and $\cong$ are essentially distinct:
\begin{sentences}\itemsep-2pt
\item
  Let $a\in X$, $|X|\geq 2$; then
  $N_X \podo K_{\{a\},X\setminus\{a\}},
  N_X \not\cong K_{\{a\},X\setminus\{a\}}$.
\item
  Let $A,B\subsetneq X$,
  $|A| = |B| >1$.
  Then $G_1 = K_A \dcup N_{X\setminus A} \cong K_B \dcup N_{X\setminus B} = G_2$.
  It is seen that 
  if $A\cup B\neq X$ then $G_1\syminus G_2\notin \bipar(X)$ and
  then $G_1 {\not\podo} G_2$.
  Take, in particular, $A = \{1,2\}$, $B = \{2,3\}$, $X = A \cup B \cup \{4\}$.
  Then $G_1 \syminus G_2$ is the $L_3$-path $1-2-3$ not in $\bipar(X)$.
\end{sentences}

The following is easy to prove
\begin{prop}\label{uzup:global}
  Let $G_1,G_2\in\grafy(X)$.
  \begin{sentences}\itemsep-2pt
  \item\label{wlas1}
    If $G_1\cong G_2$ then $\varkappa(G_1)\cong\varkappa(G_2)$.
  \item\label{wlas2}
    If $G_1\podo G_2$ then $\varkappa(G_1)\podo\varkappa(G_2)$.
  \end{sentences}
\end{prop}
\begin{proof}
  \eqref{wlas1} is evident: if a bijection $f\colon X\longrightarrow X$ maps
    ${\cal E}_1$ onto ${\cal E}_2$, where $G_i = \struct{X,{\cal E}_i}$ for $i=1,2$,
  then $f$ maps 
    $\sub_2(X)\setminus{\cal E}_1$ onto $\sub_2(X)\setminus{\cal E}_2$ 
  as well.
 \par
  Ad \eqref{wlas2}: Note that for any $G\in \grafy(X)$ we have 
  $\varkappa(G) = K_X\syminus G$. So, suppose
    $G_1\podo G_2$ i.e. $G_2 = G_1 \syminus H$
  with $H\in \bipar(X)$. Then 
  \begin{math}
    \varkappa(G_2) = G_2 \syminus K_X = (G_1 \syminus H)\syminus K_X =
    (G_1 \syminus K_X)\syminus H = \varkappa(G_1)\syminus H.
  \end{math}
  Finally, $\varkappa(G_1)\podo\varkappa(G_2)$.
\end{proof}
Clearly, {\em $\cong$ is a congruence with respect to the operation $\dcup$}: 
if $G_1\cong G'_1$
and $G_2 \cong G'_2$ (with suitably disjoint sets of vertices) then 
$G_1 \dcup G_2 \cong G'_1\dcup G'_2$; it is evident. Unhappily, 
{\em $\podo$ is not a congruence}. Indeed,
it is known that 
  $K_2\dcup K_1\podo K_3$ and $K_3\dcup K_1 \podo K_4$.
But 
  $(K_2 \dcup K_1)\dcup K_1 = K_2 \dcup N_2\podo L_4$ 
and it is known also that
  $K_4 \not\podo L_4$.

\begin{lem}
  Let $G_1,G_2,G_3\in \grafy(X)$ and 
  $G_1 \podo G_3 \cong G_2$. 
  Then there is $G'_3\in\grafy(X)$  such that 
  $G_1 \cong G'_3 \podo G_2$.
\end{lem}
\begin{proof}
  Let $f$ be a bijection that maps $G_3$ onto $G_2$,
  let $G_1 = G_3 \syminus H$ with $H\in\bipar(X)$.
  Then $f(H)\in\bipar(X)$ and
  \begin{math}
    G_2 \syminus f(H) = f(G_3)\syminus f(H) = f(G_3\syminus H) = f(G_1).
  \end{math}
  With $G'_3 = G_2 + f(H)$ we get our claim.
\end{proof}

As a consequence of the above and of \ref{prop:podo}
we get that the relation $\rowne$ defined on $\grafy(X)$ by the formula
\begin{equation}\label{def:rowne}
  G_1 \rowne G_2 \colon\iff G_1 \podo G_3 \cong G_2
  \text{ for a graph } G_3
\end{equation}
is an equivalence relation. This is our main subject of the paper.
For small values of $n = |X|$ the classification of the elements of $\grafy(X)$
with respect to the relation $\rowne$ is known (cf. \cite{pascvebl}). Clearly, 
$K_1 = N_1$ and $K_2 = L_2 \podo N_2$.
\begin{description}\itemsep-2pt
\item[$n=2$]: $\grafy(X)/\rowne \;\,= \big\{K_2,N_2\big\}$. %% $K_2 = L_2$.
\item[$n=3$]: $\grafy(X)/\rowne \;\,= \big\{ [K_3]_{\rowne},[N_3]_{\rowne} \big\}$
  (cf. \eqref{klasy:1}).
\item[$n=4$]: $\grafy(X)/\rowne \;\,= \big\{ [K_4]_{\rowne},[N_4]_{\rowne}, [L_4]_{\rowne} \big\}$.
\end{description}
The elements in the classes enumerated above are pairwise distinct.
Moreover, 
  $L_4 \rowne L_2\dcup N_2 \rowne L_3\dcup N_1$.
The case $n = 5$, which is also already solved, will be quoted in the following.

Directly from \eqref{def:rowne} and \ref{uzup:global}
we have
\begin{prop}
  Let $G_1,G_2\in\grafy(X)$. If $X_1\rowne X_2$ then 
    $\varkappa(G_1) \rowne \varkappa(G_2)$.
\end{prop}

For a graph $G = \struct{X,{\cal E}}$ and $Z\subset X$ we write
  $G\restriction Z = \struct{Z,{\cal E}\cap \sub_2(Z)}$.
As a convenient tool for determining which graphs are {\em not} $\rowne$-related we give 
\begin{lem}
  If $G_1,G_2\in\grafy(X)$, $Z\subset X$, and $G_1\podo G_2$ then
  $G_1\restriction Z \podo G_2\restriction Z$.
\end{lem}
\begin{proof}
  Let ${\cal Z}_i = {\cal E}_i\cap\sub_2(Z)$ for $i=1,2$.
  Assume that $G_1 \podo G_2$. So, there is a sequence of local complementations
  $\mu_{a_1},\ldots,\mu_{a_k}$ which composition maps ${\cal E}_1$ onto ${\cal E}_2$. 
  Note that
    if $a_j\notin Z$ then $\mu_{a_j}({\cal E}\cap \sub_2(Z)) = {\cal E}\cap\sub_2(Z)$
  and if $a_j\in Z$ then local complementation $\mu_{a_j}$ maps 
    ${\cal E}\cap \sub_2(Z)$ onto $\big(\mu_{a_j}({\cal E})\big)\cap\sub_2(Z)$,
  for every ${\cal E}\subset\sub_2(X)$.
  Consequently, a sequence of local complementations in points of $Z$ maps
  ${\cal Z}_1$ onto ${\cal  Z}_2$ and we are done
  by \ref{prop:locale}.
\end{proof}
And afterwards, as important invariants we obtain 
\begin{prop}\label{invar:struk}
  Let $G_0$ be a graph on $k$ vertices, $k < |X|$.
  For $G\in\grafy(X)$ we write
  \begin{equation}
    G/G_0 = \big\{ Z\in\sub_k(X)\colon G\restriction Z \rowne G_0 \big\}
  \end{equation}
  Let $G_1,G_2\in\grafy(X)$.
  If $G_1\rowne G_2$ then there is a bijection $f$ of $X$ which maps
  the family $G_1/G_0$ onto $G_2/G_0$.
\end{prop}
\begin{prop}\label{invar:sub}
  Let $G_0\in\grafy(Z_0)$, $|Z_0|  = k < |X|$.
  For $G\in\grafy(X)$ we define
  \begin{equation}
    \#(G;G_0) := 
    \left| G/G_0
 %%   \left\{Z \in  \sub_k(X) \colon G\restriction Z \rowne G_0   \right\}
    \right|.
  \end{equation}
  \begin{sentences}\itemsep-2pt
  \item
    Let $G_1,G_2\in\grafy(X)$.
    If $G_1\rowne G_2$ then 
    $\#(G_1;G_0) = \#(G_2;G_0)$.
  \item
    $\#(\varkappa(G);\varkappa(G_0)) = \#(G;G_0)$.
  \end{sentences}
\end{prop}
Proof is straightforward and is omitted here, but the facts formulated in 
\ref{invar:struk} and \ref{invar:sub}
will be frequently used.

The following notation will be also convenient
\begin{description}
\item[$G_0^n$] $=G_0 \dcup N_{n-k}$ where $k = |X_0|$, $G_0\in\grafy(X_0)$.
\end{description}
Analogous notation will be used when only the $\cong$-type of $G_0$ will be important,
e.g. $L_k^n$, $(K_2 \dcup K_2)^6$ etc.

\def\ilepod(#1,#2){\#\big({#1};{#2}\big)}
From \eqref{klasy:1}  we can relatively easily compute the following
formulas
\par\noindent
\begin{minipage}[m]{0.49\textwidth}
\begin{eqnarray}
\label{wzor:pody1K}
  \ilepod(L_n,K_3) & = & (n-2)(n-3)
\\ \label{wzor:pody2K}
  \ilepod(L_n,K_4) & = & \textstyle{\binom{n-3}{2}} 
\\ \label{wzor:pody3K}
  \ilepod(L_n,K_m) & = & 0 \quad\text{for }m>4
\\
\label{wzor:pody1N}
  \ilepod(L_n,N_3) & = & (n-2) + \textstyle{\binom{n-2}{3}}
\\ \label{wzor:pody2N}
  \ilepod(L_n,N_m) & = & \textstyle{\binom{n-m+1}{m}}
\end{eqnarray}
\end{minipage}
\begin{minipage}[m]{0.49\textwidth}
\begin{eqnarray}
%% \\ 
  \nonumber
  \text{ assume that} && n > 3
\\
\label{wzor:pody4K}
  \ilepod(C_n,K_3) & = & n(n-4)
\\ \label{wzor:pody5K}
  \ilepod(C_n,K_4) & = & \textstyle{\frac{n(n-5)}{2}}
\\ \label{wzory:pody6K}
  \ilepod(C_n,K_m) & = & 0 \quad\text{for }m>4
\\
\label{wzor:pody4N}
  \ilepod(C_n,N_3) & = & n + \textstyle{\frac{n(n-4)(n-5)}{6}}
\\ \label{wzor:pody6N}
  \ilepod(C_n,N_m) & = & \textstyle{\frac{n\binom{n-m-1}{m-1}}{m}}
\end{eqnarray}
\end{minipage}
\begin{proof}
  The reasoning is standard, we shall only show in several examples
  how to handle with the formulas like these.
  \par Ad \eqref{wzor:pody1K} 
  It is impossible to find $K_3$ in $L_n$. So, we search for $K_2\dcup K_1$ within $L_n$.
  These are obtained by an edge $e$ of $L_n$ and a vertex $v$ sufficiently far from the endpoints of $e$.
  If $e$ is the first or the last in the path $L_n$ then $v$ can be chosen in $n-3$ ways; 
  if $e$ is intermediate ($L_n$ contains $n-3$ such edges) then $v$ can be chosen in $n-4$ ways.
  \par Ad \eqref{wzor:pody2K}
  We must determine all the $K_2\dcup K_2$ subgraphs in $L_n$. So, we must find a pair of edges
  $\{i_1,i_1+1\}$, $\{ i_2,i_2+1\}$ such that
  $1\leq i_1,\, i_1+3 \leq i_2 \leq n-1$. Elementary combinatorics justifies that the number
  of such pairs $(i_1,i_2)$ is as claimed.
 \par Ad \eqref{wzor:pody3K}
  If a graph contains $K_A\dcup K_B$ with $|A\cup B|\geq 5$
  then it contains a point of rank at least $3$; clearly, $L_n$
  has no such a point.
  \par Ad \eqref{wzor:pody1N}
  We must determine all bipartite $K_{\{i_1\},\{ i_2,i_3 \}}$ and all $N_{i_1,i_2,i_3}$ contained in 
  $L_n$. In the first case we choose $1< i_1 <n$ (in $n-2$ ways) and set $i_2=i_1-1$, $i_3 = i_1+1$.
  In the second case we look for sequences 
  $1\leq i_1, i_1+2\leq i_2, i_2+2\leq i_3\leq n$.
\end{proof}

From this, by analogous reasonings,  we get more complex formulas,
e.g. 
\begin{multline}
 \label{wzor:pody11}
  \ilepod(L_{n_1}\dcup \ldots\dcup L_{n_k}\dcup C_{l_1}\dcup\ldots\dcup C_{l_t})^n,K_3)  = 
  \\ 
  \sum_{i=1}^{k}[(n_i-2)(n_i-3) + (n_i-1)(n-n_i)]
  + \sum_{i=1}^t[l_i(l_i-4) + l_i(n-l_i)],
\end{multline}
\begin{multline}
 \label{wzor:pody22}
  \ilepod(L_{n_1}\dcup \ldots\dcup L_{n_k}\dcup C_{l_1}\dcup\ldots\dcup C_{l_t})^n,K_4)  = 
  \\ 
  \sum_{i=1}^{k} \textstyle{\binom{n_i-3}{2}}
  +
  \sum_{j=1}^t \textstyle{\frac{l_j(l_j-5)}{2}}\,
  +
  \\
  \sum_{1\leq i_1<i_2\leq k}^t[(n_{i_1-1})(n_{i_2}-1)
  + 
  \sum_{1\leq j_1 < j_2\leq t}l_{j_1} l_{j_2}
  +
  \sum_{i=1}^k\sum_{j=1}^t (n_i - 1)l_j;
\end{multline}
we assume here that $l_1,...,l_t >3$.

%%%%%%%%%%%%%%%%%%%%%%%%%%%%%%%%%%%%%%%%%%%%%%%%%%%%%%%%%%
%%%%%%%%%%%%%%%%%%%%%%%%%%%%%%%%%%%%%%%%%%%%%%%%%%%%%%%%%% sec:induction
%%%%%%%%%%%%%%%%%%%%%%%%%%%%%%%%%%%%%%%%%%%%%%%%%%%%%%%%%%
\section{Inductive enumerating of $\rowne$-classes}\label{sec:induction}

%%\subsection{Iteration: general}

We begin with the following `inductive' observation.
Let $X_0$ be a set, $w\notin X_0$, $X = X_0\cup\{w\}$.
Then each $G\in\grafy(X)$ can be presented in the form
\begin{multline}\label{eq:indukcja}
  G = G_0 \syminus K_{w,Z}, \text{ where } Z\subset X_0,\, G_0\in\grafy(X_0),
  \\
  Z = \{ x\in X_0\colon \{ w,x \}\in G \} \text{ and }
  G_0 = G\restriction X_0.
\end{multline}
Assume that $G$ has form \eqref{eq:indukcja}.
\begin{lem}\label{lem:indukcja}
  If $G_0 \podo G'_0$, $G_0 = G'_0 \syminus K_{A,X_0\setminus A}$, 
  then $G\podo G'_0 \syminus K_{w,A\syminus Z}$.
\end{lem}
\begin{proof}
  It suffices to note that 
  $K_{A,X\setminus A}\syminus K_{A,X_0\setminus A} = K_{w,A}$
  and
  $K_{w,A}\syminus K_{w,Z} = K_{w,A\syminus Z}$.
  Then
  \begin{math}
    G = G_0 \syminus K_{w,Z} 
      = G'_0 \syminus K_{A,X_0\setminus A} \syminus K_{w,Z}
      = G'_0 \syminus K_{A,X_0\setminus A} \syminus K_{A,X\setminus A} 
        \syminus K_{A,X\setminus A} \syminus K_{w,Z}
      = G'_0 \syminus K_{w,A} \syminus K_{w,Z} \syminus K_{A,X\setminus A}
      = (G'_0 \syminus K_{w,A\syminus Z}) \syminus K_{A,X\setminus A}.
%%     = G' \syminus K_{A,X\setminus A},
  \end{math}
%%  where
%%    $G' = G'_0 \syminus K_{w,Z\setminus A}$.
\end{proof}
In consequence,
to determine all the types of graphs on $X$ it suffices to choose a point
$w \in X$, and for each type $G_0$ of a graph on $X_0 = X\setminus \{w\}$
enumerate all, up to an isomorphism of $G_0$, $k$-subsets $Z$ of $X_0$
such that $2 k \leq |X|$. Each type of a graph on $X$ is realized as 
$G_0 \syminus K_{w,Z}$ with so obtained $G_0$'s and $Z$'s.
To complete the task it suffices to verify which of the graphs on the list
composed so far are $\rowne$-equivalent and which are not.
Let us illustrate how this procedure works and let us apply it to the case $n=|X| = 6$.

\newcounter{graftyp}\setcounter{graftyp}{0}
\def\refive#1{5:\ref{typ5:#1}}
\def\refsix#1{6:\ref{typ6:#1}}
\let\resix\refsix

Let us quote the following
\begin{prop}[\normalfont{\cite[page 204]{pascvebl}}]\label{Klasyf:5}
  There are exactly 7 $\rowne$-types of graphs on $5$ vertices.
  These are the following: \setcounter{graftyp}{0}

\begin{center}
\begin{tabular}{llll}
  \refstepcounter{graftyp}5:\thegraftyp\label{typ5:1} $K_5$, $\podo K_4^5$;
&
  \refstepcounter{graftyp}5:\thegraftyp\label{typ5:2} $N_5 = \varkappa(K_5)$;
&
  \refstepcounter{graftyp}5:\thegraftyp\label{typ5:3} $C_5$;
&
  \refstepcounter{graftyp}5:\thegraftyp\label{typ5:4} $L_2^5 = K_2^5$;
\\
  \refstepcounter{graftyp}5:\thegraftyp\label{typ5:5} $\varkappa(L_2^5)$, $\podo K_3^5$;
&
  \refstepcounter{graftyp}5:\thegraftyp\label{typ5:6} $L_3^5$;
&
  \refstepcounter{graftyp}5:\thegraftyp\label{typ5:7} $\varkappa(L_3^5) \podo L_5$.
\end{tabular}
\end{center}
\iffalse
  \begin{enumerate}[({5:}1)]\itemsep-2pt
  \item\label{typ5:1} $K_5$, $\podo K_4^5$
  \item\label{typ5:2} $N_5 = \varkappa(K_5)$,
  \item\label{typ5:3} $C_5$,
  \item\label{typ5:4} $L_2^5 = K_2^5$,
  \item\label{typ5:5} $\varkappa(L_2^5)$, $\podo K_3^5$
  \item\label{typ5:6} $L_3^5$,
  \item\label{typ5:7} $\varkappa(L_3^5) \podo L_5$.
  \end{enumerate}
\fi
\end{prop}

\def\dodajkrop{%
\let\xtheenumii\theenumii
\def\theenumii{.\xtheenumii}}

Our goal (one of some) is to prove the following
\begin{thm}[$6$-graphs]\label{Klasyf:6}
  Let $G$ be a graph on $6$ vertices. Then $G$ is $\rowne$-equivalent to one of
  the following graphs. \setcounter{graftyp}{0}

\begin{center}
\begin{tabular}{llll}
  \refstepcounter{graftyp}6:\thegraftyp\label{typ6:1}  $K_6$; %% 1.0
&
  \refstepcounter{graftyp}6:\thegraftyp\label{typ6:2}  $N_6$; %% 2.0
&
  \refstepcounter{graftyp}6:\thegraftyp\label{typ6:3}  $L_2^6$; %% 2.1
&
  \refstepcounter{graftyp}6:\thegraftyp\label{typ6:4}  $\varkappa(L_2^6)$, $\podo K_4^6,\, C_3\dcup L_3$; %% 1.1
\\
%%  \refstepcounter{graftyp}6:\thegraftyp\label{typ6:5}  $C_5^6$; %% 3.1
%% &
  \refstepcounter{graftyp}6:\thegraftyp\label{typ6:6}  $(L_2\dcup L_2)^6$; %% $\podo \varkappa(L_6)$;
&
  \refstepcounter{graftyp}6:\thegraftyp\label{typ6:5}  $C_5^6$; %% 3.1
&
  \refstepcounter{graftyp}6:\thegraftyp\label{typ6:7}  $L_3^6$; %% $\podo \varkappa(L_6)$;
&
  \refstepcounter{graftyp}6:\thegraftyp\label{typ6:8}  $K_3^6$; %% $\podo \varkappa(L_6)$;
\\
  \refstepcounter{graftyp}6:\thegraftyp\label{typ6:9}  $(L_3 \dcup L_2)^6$;
&
  \refstepcounter{graftyp}6:\thegraftyp\label{typ6:10}  $L_4^6$; %% $\podo \varkappa(L_6)$;
&
  \refstepcounter{graftyp}6:\thegraftyp\label{typ6:12}  $L_3\dcup L_3$, %% $\podo \varkappa(L_6)$;
&
  \refstepcounter{graftyp}6:\thegraftyp\label{typ6:11} $\varkappa(L_3^6)$; $\podo (C_3\dcup L_2)^6$
%% &
%%  \refstepcounter{graftyp}6:\thegraftyp\label{typ6:12}  $L_3\dcup L_3$, %% $\podo \varkappa(L_6)$;
\\
  \refstepcounter{graftyp}6:\thegraftyp\label{typ6:13}  $C_4^6$, $\podo \varkappa(C_6)$;
&
  \refstepcounter{graftyp}6:\thegraftyp\label{typ6:14} $L_5^6$;
&
  \refstepcounter{graftyp}6:\thegraftyp\label{typ6:15} $L_6$;
&
  \refstepcounter{graftyp}6:\thegraftyp\label{typ6:16} $C_6$, $\podo\varkappa(L_2\dcup L_2\dcup L_2)$.
\end{tabular}
\end{center}
\iffalse
  \begin{enumerate}[{6:}1]\itemsep-2pt
  \item\label{typ6:1}  $K_6$;
  \item\label{typ6:2}  $N_6$; 
  \item\label{typ6:3}  $L_2^6$;
  \item\label{typ6:4}  $\varkappa(L_2^6)$
  \item\label{typ6:5}  $C_5^6$;
  \item\label{typ6:6}  $L_4^6$, $\podo \varkappa(L_6)$;
  \item\label{typ6:7}  $L_3^6$;
  \item\label{typ6:8}  $K_3^6$;
  \item\label{typ6:9}  $(L_3 \dcup L_2)^6$;
  \item\label{typ6:10} $\varkappa(L_3^6)$;
  \item\label{typ6:11} $L_5^6$;
  \item\label{typ6:12} $L_6$;
  \item\label{typ6:13} $C_6$.
  \end{enumerate}
\fi
  %
  No two graphs in this list are $\rowne$-equivalent. 
\end{thm}
\begin{proof}
  Let us assume that graphs classified here are defined on the set $\{ 1,2,3,4,5,6 \}$,
  write $w = 6$ and apply \ref{lem:indukcja}. Let $X_0 = \{ 1,2,3,4,5 \}$.
  So, we obtain the list of graphs of the form $G_0\syminus K_{w,Z}$, 
  $G_0$ is one from among those enumerated in \ref{Klasyf:5} 
  and $Z\in\sub_1(X_0)\cup\sub_2(X_0)\cup\{ \emptyset\}$.
  In what follows we shall indicate mainly sets of edges of corresponding graphs.
  \par
  To shorten notation we shall also write ``i.j = $\varkappa$(i',j')",
  if $G = \varkappa(G')$, $G$ stands on the position i.j in the list below, and $G'$ has
  the position i'.j'.
  Analogous meaning has notation ``i.j = i'.j'", ``$G$ = i.j", ``i.j$\;\rowne\;$i'.j'", and ``$G \rowne$ i.j".
  The symbol {\small$\bigcirc$} indicates the case when the resulting graph already belongs to those enumerated through
  \resix{1}-\resix{16} or it coincides with a graph considered earlier.
  So, it means `there is nothing to prove in this case'.
 %% LISTING 
 \begin{enumerate}[1.]\itemsep-2pt
  \item Let $G_0$ in \refive{1}
    \begin{enumerate}[{\theenumi}1]\itemsep-2pt\setcounter{enumii}{-1}\dodajkrop
    \item\label{g1:0} 
      $Z=\emptyset$. Then  
      $G = \sub_2(X_0) = K_{X_0} \dcup K_{\{ 6 \}} \rowne K_X$, so 
      $G$ has the type $K_6$, declared in \resix{1}.\myend
    \item\label{g1:1}  
      $Z = \{ 5 \}$. Then
      $G = \sub_2(X_0) \cup \{ \{ 5,6 \} \}$ and
      $G \rowne \varkappa(L_2^6)$, declared in \resix{4}. \myend
      \\
      Moreover, $G \rowne C_3 \dcup L_3$. 
    \item\label{g1:2}  
      $Z = \{ 4,5 \}$. Then
      $G = \sub_2(X_0) \cup \{ \{ 5,6 \}, \{ 4,6 \} \}$
      and $G \rowne $ \ref{g5:1}, \ref{g7:5}.
    \end{enumerate}
  \item Let $G_0$ in \refive{2}
    \begin{enumerate}[{\theenumi}1]\itemsep-2pt\setcounter{enumii}{-1}\dodajkrop
    \item\label{g2:0}  
      $Z = \emptyset$. Then
      $G = \emptyset$, so $G$ has the type $N_6$, declared in \resix{2}.\myend
    \item\label{g2:1}  
      $Z = \{ 5 \}$. Then
      $G = \{ \{ 5,6 \} \}$, so $G$ has the type $L_2^6$, declared in \resix{3}.\myend
    \item\label{g2:2}  
      $Z = \{ 4,5 \}$. Then
      $G = \{ \{ 5,6 \},  \{ 4,6 \} \}$, so $G$ has the type $L_3^6$, declared in \resix{7}. \quad\strut\myend
    \end{enumerate}
  \item Let $G_0$ in \refive{3}. Say, $G$ is the cycle $1-2-3-4-5-1$. 
    Considering the automorphism group of $G_0$ 
    we  see that it suffices to consider the following cases only.
    \begin{enumerate}[{\theenumi}1]\itemsep-2pt\setcounter{enumii}{-1}\dodajkrop
    \item\label{g3:0} 
      $Z = \emptyset$.
      Then $G$ has the type $C_5^6$, declared in \resix{5}.\myend
    \item\label{g3:1} 
      $Z = \{ 1 \}$.
      Then $G \rowne $ \ref{g6:5}
    \item\label{g3:2}     
      $Z = \{ 1,2 \}$.
      Then $G \rowne $ \ref{g7:1}.
    \item\label{g3:3} 
      $Z = \{ 1,3 \}$.
      Then $G\rowne $ \ref{g4:5}
    \end{enumerate}
  \item Let $G_0$ in \refive{4}. Say, the unique edge of $G_0$ is $\{ 1,2 \}$. It is seen that
    it suffices to consider the following sets $Z$.
    \begin{enumerate}[{\theenumi}1]\itemsep-2pt\setcounter{enumii}{-1}\dodajkrop
    \item\label{g4:0} 
      $Z = \emptyset$. 
      Then $G$ has the type $L_2^6$, equal to \ref{g2:1}.\myend
    \item\label{g4:1} 
      $Z = \{ 1 \}$. 
      Then $G$ is, in fact, the path $6-1-2$, so it has the type $L_3^6$
      (and $G = $ \ref{g2:2}).\myend
    \item\label{g4:2} 
      $Z = \{ 3 \}$. 
      Then $G$ has two, disjoint, edges: 
      it is $(L_2\dcup L_2)^6$, declared in \resix{6}.\myend
    \item\label{g4:3} 
      $Z = \{ 1,2 \}$. 
      Then $G$ is the triangle $1,2,6$, so it is $K_3^6$, 
      declared in \resix{8}.\myend
    \item\label{g4:4} 
      $Z = \{ 3,4 \}$. 
      Then $G$ consists of the 3-path $3-6-4$ and the edge $1-2$ 
      and thus it is $(L_3 \dcup L_2)^6$, declared in \resix{9}.\myend
    \item\label{g4:5} 
      $Z = \{ 1,3 \}$. 
      Then $G$ consists of the $4$-path
      $2-1-6-3$; it is $L_4^6$, declared in \resix{10}.\myend
    \end{enumerate}
  \item Let $G_0$ in \refive{5}.
    Then $G_0 = \varkappa(G'_0)$, where $G'_0$ is given in \refive{4}.
    Consequently, $G \rowne \varkappa(G')$, where $G'$ is a one among those
    enumerated through \ref{g4:0}--\ref{g4:5}.
    \begin{enumerate}[{\theenumi}1]\itemsep-2pt\setcounter{enumii}{-1}\dodajkrop
    \item\label{g5:0} 
      $G = \varkappa(L_2^6)$. $G =$ \ref{g1:1}.\myend
    \item\label{g5:1} 
      $G = \varkappa(L_3^6)$, declared in \resix{11}.\myend
      \\
      Moreover, 
      $G \rowne (C_3 \dcup L_2)^6$.
    \item\label{g5:2} 
      $G = \varkappa((L_2\dcup L_2)^6) \rowne $ \ref{g6:4}.
    \item\label{g5:3} 
      $G = \varkappa(K_3^6) \rowne $ \ref{g4:3}.
    \item\label{g5:4} 
      $G = \varkappa((L_3\dcup L_2)^6) \rowne$ \ref{g4:4}.
    \item\label{g5:5} 
      $G = \varkappa(L_4^6) \rowne $ \ref{g7:1}; moreover $G \rowne L_2\dcup L_4$. 
    \end{enumerate}
  \item Let $G_0$ in \refive{6}. One can assume that $G_0$ is the path $1-2-3$, and $4,5$ are 
    isolated. The following cases must be considered.
    \begin{enumerate}[{\theenumi}1]\itemsep-2pt\setcounter{enumii}{-1}\dodajkrop
    \item\label{g6:0} 
     $Z= \emptyset$.
     Then $G = L_3^6$ (as in \ref{g4:1} and \ref{g2:2}). \myend
    \item\label{g6:1}
     $Z = \{ 4 \}$. Then $G$ consist of the 3-path $1-2-3$ and the edge $4-6$, so 
     $G = (L_2\dcup L_3)^6$, as in \ref{g4:4}.\myend
    \item\label{g6:2}
     $Z = \{ 1 \}$.
     Then $G$ is the path $6-1-2-3$, so $G = L_4^6 = $ \ref{g4:5}.\myend
    \item\label{g6:3}
     $Z = \{ 2 \}$; then $G \approx $ \ref{g4:1}.
    \item\label{g6:4}
     $Z = \{ 4,5 \}$.
     Then $G$ consist of two disjoint 3-paths $1-2-3$ and $4-6-5$, so 
     $G = L_3\dcup L_3$, declared in \resix{12}.\myend
    \item\label{g6:5}
     $Z = \{ 4,1 \}$.
     Then $G$ is the path $4-6-1-2-3$, so $G = L_5^6$, declared in \resix{14}.\myend
    \item\label{g6:6}
     $Z = \{ 4,2 \}$; then $G \rowne$ \ref{g4:5}.
    \item\label{g6:7}
     $Z = \{ 1,3 \}$.
     Then $G$ is the closed cycle $5-1-2-3-5$, so 
     $G = C_4^6$, declared in \resix{13}.\myend
    \item\label{g6:8}
     $Z = \{ 1,2 \}$. Then $G \rowne $ \ref{g4:4}.
    \end{enumerate}
  \item Let $G_0$ in \refive{7}.\label{cas:g7}
    One can assume that $G_0$ is the 5-path $1-2-3-4-5$. The following sets $Z$ must be 
    considered.
    \begin{enumerate}[{\theenumi}1]\itemsep-2pt\setcounter{enumii}{-1}\dodajkrop
    \item\label{g7:0}
      $Z = \emptyset$.
      Then $G = L_5^6$ = \ref{g6:5}.\myend
    \item\label{g7:1}
      $Z = \{ 1 \}$. Then $G$ is the path 
      $6-1-2-3-4-5$, so $G = L_6$, declared in \resix{15}.\myend
    \item\label{g7:2}
      $Z = \{ 2\}$. Then $G \rowne $ \ref{g6:1} and 
      $G \rowne $ \ref{g7:6}.
    \item\label{g7:3}
      $Z = \{ 3 \}$. Then $G \rowne $ \ref{g3:1}.
    \item\label{g7:4}
      $Z = \{ 1,5 \}$.
      Then $G$ is the closed cycle 
      $6-1-2-3-4-5-6$, i.e. $G = C_6$, declared in \resix{16}. \myend
    \item\label{g7:5}
      $Z = \{ 1,2 \}$: cf.  \ref{g1:2}.\myend
    \item\label{g7:6}
      $Z = \{ 1,3 \}$: cf. \ref{g7:2}.\myend
    \item\label{g7:7}
      $Z = \{ 1,4 \}$:  Then $G =$  \ref{g3:1}. 
    \item\label{g7:8}
      $Z = \{ 2,3 \}$. Then $G \rowne $ \ref{g7:1}.
    \item\label{g7:9}
      $Z = \{ 2,4 \}$: Then $G \rowne$ \ref{g4:2}.
    \end{enumerate}
  \end{enumerate}
 %% PROOFS of REDUCTION
  Ad  \ref{g1:1}
    Note that $G \syminus K_{\{ 1,5 \},\{ 2,3,4,6\}} \cong  C_3 \dcup L_3$.
    Moreover, $G \syminus K_{\{ 6 \},\{ 1,2,3,4,5\}} \cong \varkappa(L_2^6)$.
  \par Ad \ref{g1:2}
    $G \syminus K_{\{ 1,5 \},\{ 2,3,4,6\}}$ is the path $5-1-6-4$ connected with
    the triangle $4,3,2$ i.e. $G \rowne$  \ref{g7:5}.
    Next, 
    $G \syminus K_{\{ 6 \},\{ 1,2,3,4,5\}} = \varkappa(L_3^6) = \varkappa(\text{\ref{g4:1}})$
    = \ref{g5:1}.
  \par Ad \ref{g3:1}
    $G \syminus K_{\{ 1,4 \},\{ 2,3,5,6\}}$ is the path 
    $1-3-2-4-6$ i.e. $L_5^6$ = \ref{g6:5}.
  \par Ad \ref{g3:2}
    $G \syminus K_{\{ 1,3 \},\{ 2,4,5,6\}}$ is the path $1-4-5-3-6-2$ i.e. $= L_6$ = \ref{g7:1}.
  \par Ad \ref{g3:3}
    $G \syminus K_{\{ 1,3 \},\{ 2,4,5,6\}}$ is the $L_4$-path $1-4-5-3$ = \ref{g4:5}.
  \par Ad \ref{g5:1}
    $G \syminus K_{\{ 3,4,5 \},\{ 1,2,6\}}$ is the union of one edge $2-6$ and the triangle
    $3,4,5$, so it is $(L_2\dcup C_3)^6$.
  \par Ad \ref{g5:2}
    $G \syminus K_{\{ 1,2,4 \},\{ 3,5,6\}}$ is the union of two $L_3$-paths:
    $1-4-2$ and $3-5-6$ = \ref{g6:4}.
  \par Ad \ref{g5:3}
    $G \syminus K_{\{ 1,2,6 \},\{ 3,4,5\}}$ is the triangle $3,4,5$, so we obtain
    $K_3^6 =$ \ref{g4:3}.
  \par Ad \ref{g5:4}
    $G \syminus K_{\{ 1,2,5 \},\{ 3,4,6 \}}$ is the union of the edge $3-4$ and 
    the $L_3$-path $1-5-2$.
  \par Ad \ref{g5:5}
    $G \syminus K_{\{ 1,2,4 \},\{ 3,5,6 \}}$ is the path $2-4-1-6-5-3$ = \ref{g7:1}.
    Moreover, 
    $G \syminus K_{\{ 4,5 \},\{ 1,2,   3,6\}}$ is the union of the
    $L_4$-path $1-3-2-6$ and the edge $4-5$, so $G \rowne L_2\dcup L_4$.
  \par Ad \ref{g6:3}
    $G \syminus K_{\{ 2 \},\{ 1,3,4,5,6\}}$ is the $L_3$-path $4-2-5$.
  \par Ad \ref{g6:6}
    $G \syminus K_{\{ 2 \},\{ 1,3,4,5,6\}}$ is the $L_4$-path $5-2-4-6$.
  \par Ad \ref{g6:8}
    $G \syminus K_{\{ 2 \},\{ 1,3,4,5,6\}}$ is the union of the $L_3$-path
    $4-2-5$ and the edge $1-6$.
  \par Ad \ref{g7:2}
    $G \syminus K_{\{ 2,4 \},\{ 1,3,5,6\}}$ is the union of the $L_3$-path $1-4-6$
    and the edge $2-5$, i.e = \ref{g4:4}.
    Next,
    $G \syminus K_{\{ 3 \},\{ 1,2,4,5,6\}}\cong $ \ref{g7:6}.
  \par Ad \ref{g7:3}
    $G \syminus K_{\{ 1,4,6 \},\{ 2,3,5\}}$ is the $5$-cycle $1-3-2-6-5-1$
    with the edge $2-4$ added, so it is $\cong $ \ref{g3:1}.
  \par Ad \ref{g7:8}
    $G \syminus K_{\{ 1,4,6 \},\{ 2,3,5 \}}$ is the $L_6$-path 
    $4-2-3-1-5-6$.
  \par Ad \ref{g7:9}
    $G \syminus K_{\{ 2,4 \},\{ 1,3,5,6\}}$ is the union of two disjoint 
    edges $1-4$ and $2-5$, which is = \ref{g4:2}.

 %% DISTINCTIONS
  To complete the proof we analyse Table \ref{tab:paramy6}
  (note that $\#(G;N_3) = 20 - \#(G;K_3)$, as $G$ contains 20 subgraphs on 3 vertices). 
  From \ref{invar:sub} we see that only two cases must be distinguished
  by other methods. Here we apply \ref{invar:struk}.
  To distinguish \resix{10} and \resix{13} we note  that
  $L_4^6/N_4$ consists of three 4-sets with a common 2-set; while no such a
  common subset exists for the elements of $C_4^6/N_4$.
  Similarly, to distinguish \resix{15} and \resix{16} we note that
  $L_6/K_4$ consists, analogously, of three subsets with the common 2-set, and this is not true 
  for $C_6/K_4$.
\end{proof}

\begin{table}
\begin{tabular}{r|cccccccccc}
$G$ & \resix{1} & \resix{2} & \resix{3} & \resix{4} & \resix{5} & \resix{6} & \resix{7} & \resix{8} & \resix{9} & \resix{10} 
%% \resix{11} & \resix{12} & \resix{13} & \resix{14} & \resix{15} & \resix{16}
\\
$\#(G;K_3)$ &
%% 1   2   3    4    5     6    7    8     9      10
  20 & 0 & 4  & 16 & 10  & 8  & 6  & 10  & 10  & 8
\end{tabular}

\strut\hfill
\begin{tabular}{r|cccccc}
$G$ & 
%% \resix{1} & \resix{2} & \resix{3} & \resix{4} & \resix{5} & \resix{6} & \resix{7} & \resix{8} & \resix{9} & \resix{10} & 
\resix{11} & \resix{12} & \resix{13} & \resix{14} & \resix{15} & \resix{16}
\\
$\#(G;K_3)$ &
%% 11     12    13    14     15     16
   14   & 12  & 8   & 10   & 12   & 12
\\
\hline
\end{tabular}

\begin{center}
\begin{tabular}{r||cccc|ccc|ccc}
$G$ & 
\resix{5} & \resix{8} & \resix{9} & \resix{14} & \resix{6} & \resix{10} & \resix{13} & \resix{12} & \resix{15} & \resix{16} 
\\
$\#(G;K_4)$ &
%% 5   8   9    14    6    10    13    12     15    16
   0 & 3 & 2  & 1   & 1  & 0   & 0   &  4   & 3   &  3
\\ 
$\#(G;N_4)$ &
%% 5   8    9    14   6    10    13    12    15    16
   0 & 3  & 2 &  1  & 4  &  3  & 3  &   1  &  0  &  0
\\ \hline
\end{tabular}
\end{center}
\caption{Parameters $\#(G;K_3)$, $\#(G;K_4)$, and $\#(G;N_4)$ of
the graphs $G$ defined in \ref{Klasyf:6}}
\label{tab:paramy6}
\end{table}

%%%%%%%%%%%%%%%%%%%%%%%%%%%%%%%%%%%%%%%%%%%%%%%%%%%%%%%%%%
%%%%%%%%%%%%%%%%%%%%%%%%%%%%%%%%%%%%%%%%%%%%%%%%%%%%%%%%%% biblio
%%%%%%%%%%%%%%%%%%%%%%%%%%%%%%%%%%%%%%%%%%%%%%%%%%%%%%%%%%

\bigskip

\par\noindent\small
Authors' address:\\
El{\.z}bieta B{\l}aszko, 
Ma{\l}gorzata Pra{\.z}mowska, Krzysztof Pra{\.z}mowski\\
Institute of Mathematics, University of Bia{\l}ystok\\
K. Cio{\l}kowskiego 1M\\
15-245 Bia{\l}ystok, Poland\\
e-mail: 
{\ttfamily e.blaszko@wp.pl}, 
{\ttfamily malgpraz@math.uwb.edu.pl},
{\ttfamily krzypraz@math.uwb.edu.pl}

\end{document}